\documentclass{article}
\usepackage[utf8]{inputenc}

\usepackage{lineno}
\usepackage{amsmath,amsthm,amssymb}
\usepackage{graphicx}
\usepackage{comment}
\usepackage{anyfontsize}
\usepackage{multirow}
\usepackage{wrapfig}
\usepackage{listings}
\usepackage{tcolorbox}
\usepackage[colorlinks, bookmarks]{hyperref}
\usepackage[english]{babel}
\newtheorem{theorem}{Theorem}[section]
\newtheorem{lemma}[theorem]{Lemma}
\newtheorem{cor}[theorem]{Corollary}
\newtheorem{corollary}[theorem]{Corollary}

\theoremstyle{definition}

\newtheorem{rmk}[theorem]{Remark}
\newtheorem{defin}[theorem]{Definition}

\def\C{{\mathbb C}}
\def\K{{\mathbb K}}

\def\diag{{\rm diag\,}}

\begin{document}
	\title{Integrability of matrices}
	\author{S. Danielyan$^{a}$, A. Guterman$^{a,d,e}$, E. Kreines$^{b,c,e}$,
		F. Pakovich$^{c}$}

	\date{\small $^{a}$Bar-Ilan University, Ramat Gan 5290002, Israel \\ $^{b}$Tel Aviv University, Tel Aviv 6997801,
		Israel \\ $^{c}$Ben Gurion University of the Negev, Beer-Sheva 8410501, Israel\\
		$^{d}$Lomonosov Moscow State University, Moscow
		119991, Russia \\ $^{e}$Moscow Center of Fundamental and Applied Mathematics,
		Moscow 119991, Russia }
	\maketitle
	
	\begin{abstract}
		The concepts of differentiation and integration for matrices are known.
		As far as each matrix is differentiable, it is 
		not clear a priori whether 
		a given matrix is integrable or not.
		Recently some progress was obtained for  diagonalizable matrices, however general problem remained open. 
		In this paper, we present a full solution 
		of the integrability problem. Namely,  we provide necessary and sufficient conditions for a given matrix to be integrable in terms of its characteristic polynomial. Furthermore, we  find necessary and sufficient conditions for the existence of integrable and non-integrable matrices with given geometric multiplicities of eigenvalues. Our approach relies on properties of some special classes of polynomials, namely, Shabat polynomials and conservative polynomials, arising in number theory and dynamics. 
		\\
		MSC2020: 15A20
		\\
		{\bf Keywords}: polynomials, matrices, differentiators, integrators
	\end{abstract}

	
	\section{Introduction}
	
	The concepts of differentiation and integration for matrices were introduced
	for studying zeros and critical points of complex polynomials.
	The 
	notion of matrix   differentiability was
	introduced by Davis in \cite{Davis} and further  investigated
	in~\cite{IntegratorsOfMatricies,CheN06,CheN10,KushT16,Differentiators}. 
	The converse operation, namely, the operation of integration, is due to
	Bhat and    Mukherjee (\cite{IntegratorsOfMatricies}).
	Notice that the inegrability of matrices has  applications to the nonnegative inverse eigenvalue problem and to   inequalities like the dual Schoenberg type inequality (see  \cite{DanielyanGuterman,HMcCPT}  and references therein).

	Originally, differentiability and integrability of matrices  were defined for matrices over the field of complex numbers $\C$. In this paper, we allow the main field $\K$ to be an arbitrary algebraically
	closed subfield of $  \C$. Notice that such $\K$ contains $\mathbb Q$ and its algebraic closure $\overline {\mathbb Q} $. We denote
	the set of $n\times n$ matrices over $\K$ by $M_n := M_n(\K)$, the set of row vectors of the length $n$ over $\K$ by $\K^n$, and the set of polynomials in $x$ over $\K$ by $\K[x]$.  Accordingly, our main definition is the following.

	\begin{defin}
		Let  $B \in M_{n}(\K).$ Then $ A \in M_{n+1}(\K)$ is called
		{\it an integral  of } $B$
		if there exist $u,v\in \K^n$ and $b\in \K$ such that
		$
		A =
		\left[
		\begin{array}{cccc}
			B & u^{\top} \\
			v & b 
		\end{array}
		\right]
		$
		and   $p_{B}(x) = \frac{1}{n+1}p_{A}'(x)$,  where   $p_M(x)$ denotes the characteristic polynomial of a matrix~$M$, and $p'$ is the formal derivative of $p$. 
		We say that $B$ is integrable if there exists an integral~of~$B$. 
	\end{defin}

	Integrability of matrices was firstly investigated
	in~\cite{IntegratorsOfMatricies}. 
	In particular,
	it was proved in~\cite{IntegratorsOfMatricies} that non-derogatory matrices, that is, matrices that
	have only one Jordan cell for each eigenvalue, are always integrable. Thus, the remaining problem was to determine if a matrix having several Jordan cells with the same eigenvalue is integrable or not.
	In the subsequent paper \cite{DanielyanGuterman}, the integrability problem was solved for
	diagonalizable matrices by 
	using  methods of matrix analysis
	and linear algebra.
	However, the general case remained open. 
	In this paper, we provide a  complete solution of the integrability problem  for  arbitrary matrices.

	\section{Main results and plan of the paper}

	Our first main result is the following statement.

	\begin{theorem} \label{Cor:THM:general_integrable_iff_fullintegral}
		Let $B\in M_n$ be a matrix, and $S = \{\lambda \ |\ \dim \ker(B-\lambda I)>1 \}$ the set of eigenvalues of $B$ such that  there exist more than one Jordan cells with  the eigenvalue $\lambda $. 
		Then $B$ is integrable if and only if an integral $\int p_B(x) {\rm d}x$ takes the same value on all elements of~$S$.
	\end{theorem}
	
	Solving the integrability problem for any given matrix, Theorem \ref{Cor:THM:general_integrable_iff_fullintegral} also reduces finding 
	necessary and sufficient conditions for the existence of integrable and non-integrable matrices with given geometric multiplicities of eigenvalues to solving the corresponding problem about integrals of polynomials.  
	In more detail, let us  introduce the concept of  $S$-full integral of a polynomial by the following definition.

	\begin{defin}
		Let $f(x)\in \K[x]$ be a polynomial,  $Z'(f)$ the set of zeros
		of $f$ of multiplicity at least two, and $S$ a subset of $ Z'(f)$.
		A polynomial $F(x)\in \K[x]$ is called an $S$-full integral of $f(x)$ if
		
		1. $F'(x) = f(x)$,
		
		2. $F(t) = 0$ for any $t \in S.$
		
		\noindent We say that $f(x)$ is $S$-full integrable if there exists an $S$-full integral~of~$B$. 
		
	\end{defin}

	In view of Theorem \ref{Cor:THM:general_integrable_iff_fullintegral},  the existence problem for  integrable and non-integrable matrices with given geometric multiplicities of eigenvalues is equivalent to the existence problem for $S$-full integrals of a polynomial with given multiplicities of zeros, and much of the paper is devoted to the solution of the last problem.  
	A particular case of this problem with $S = Z'(f)$ was solved   in the recent paper  \cite{dg}. 
	As a corollary, 
	the relations between the
	spectrum of a diagonalizable matrix and
	its integrability  
	were established  (see~\cite[Theorem~3.13]{DanielyanGuterman}).
	
	In this paper, we solve the existence problem for $S$-full integral of a polynomial with given multiplicities of zeros. 
	We write the polynomials under consideration in the form
	\begin{equation} \label{eq:f}
		f(x) = (x - a_1)^{\alpha_1}\ldots(x - a_m)^{\alpha_m}
		(x - b_1)^{\beta_1}\ldots(x - b_{k})^{\beta_{k}},
	\end{equation}
	where the zeros $a_1,\ldots, a_m, b_1, \ldots, b_k$ are pairwise distinct,
	and the multiplicities satisfy the conditions
	$\alpha_1,\ldots, \alpha_{m}\geq 2,\,\beta_1,\ldots, \beta_{k}\ge 1$,
	so that the zeros $ a_1,\ldots, a_m$ are multiple, and the zeros $b_1, \ldots, b_k$
	can be arbitrary.
	Thus, the set $S=\{a_1,\ldots, a_m\}$ is always a subset of $Z'(f)$,
	and we can consider the question about $S$-full integrability of~$f$.

	Our main result concerning the existence of $S$-full integrals is  the following.
	
	
	\begin{theorem} \label{THM:main} 
		Let  $m,k\ge 0$,
		$\alpha_1,\ldots, \alpha_{m}\geq 2,$ $\beta_1,\ldots, \beta_{k}\ge 1$ be integers, 
		and 
		\[
		n=\sum\limits_{j=1}^{m}\alpha_j+\sum\limits_{i=1}^{k}\beta_i, \ \ \ \ M=\sum\limits_{j=1}^{m}\alpha_j.
		\]
		Then 
		the following statements are true.
		
		\begin{enumerate} 
			\item If $m=0$, then for all pairwise distinct
			$b_1,\ldots, b_k \in \K$  the polynomial
			\[
			f(x) = (x - b_1)^{\beta_1}\ldots(x - b_k)^{\beta_k}
			\]
			has an $S$-full integral with respect to~$S=\emptyset$.
			
			\item If $m=1$, then for all pairwise distinct
			$a, b_1,\ldots, b_k \in \K$  the polynomial
			\[
			f(x) = (x - a)^{\alpha_1}(x - b_1)^{\beta_1}\ldots(x - b_k)^{\beta_k},
			\]
			has an $S$-full integral with respect to~$S=\{a\}$.
			
			\item If\, $2 \leq m \leq n - M + 1$, then there exist pairwise distinct
			$a_1,\ldots, a_m, b_1, \ldots, b_k  \in \overline{\mathbb Q}$
			such that
			the polynomial $f_1(x) = (x - a_1)^{\alpha_1}\ldots(x - a_m)^{\alpha_m}
			(x - b_1)^{\beta_1}\ldots(x - b_{k})^{\beta_{k}}$  has an
			$S$-full integral with respect to $S=\{a_1,\ldots,a_m\}$, 
			and  there exist pairwise distinct $ 
			\{a_1',\ldots, a_m', b_1', \ldots, b_k'\} \in \overline{\mathbb Q}$
			such that
			the polynomial $f_2(x) = (x - a_1')^{\alpha_1}\ldots(x - a_m')^{\alpha_m}
			(x - b_1')^{\beta_1}\ldots(x - b_{k}')^{\beta_{k}}$ does not have an
			$S$-full integral with respect to $S=\{a_1',\ldots,a_m'\}$.

			\item If $m > n - M + 1$ then for all pairwise distinct
			$a_1,\ldots, a_m, b_1, \ldots, b_k \in \K$  the polynomial
			$$f(x) = (x - a_1)^{\alpha_1}\ldots(x - a_m)^{\alpha_m}
			(x - b_1)^{\beta_1}\ldots(x - b_k)^{\beta_{k}}$$
			does not have an
			$S$-full integral with respect to $S=\{a_1,\ldots,a_m\}$.
		\end{enumerate}
		In particular,  an $S$-full integrable polynomial of the form~\eqref{eq:f}  exists if and only if $m\le n-M+1$.

	\end{theorem}

	We remark that the   main difficulty (as well as the main interest) 
	in proving Theorem \ref{THM:main} is to establish the existence of integrable and non-integrable polynomials in the third part.
	Our approach to the proof of this part relies on some
	techniques having their origin in number theory and dynamics. 
	Namely,  we use the beautiful relations between plane trees and two types of
	complex polynomials: Shabat polynomials arising in the 
	Grothendieck ``dessin d'enfant'' theory and conservative polynomials arising in dynamics. To construct polynomials that posses  $S$-full integrals we use Shabat polynomial, while 
	to construct polynomials that do not possess $S$-full integrals
	we use conservative polynomials. 
	
	Finally, combining Theorem~\ref{Cor:THM:general_integrable_iff_fullintegral} and Theorem  \ref{THM:main} we obtain the following result.

	\begin{theorem} \label{THM:matrix_int_classification}
		Let  $m,k\ge 0$ and
		$\alpha_1,\ldots, \alpha_{m}\geq 2,\,\beta_1,\ldots, \beta_{k}\ge 1$
		be  integers,
		$$
		n=\sum\limits_{j=1}^{m}\alpha_j+\sum\limits_{i=1}^{k}\beta_i, \ \ \ \
		M=\sum\limits_{j=1}^{m}\alpha_j,
		$$
		and 
		$\cal M$ the subset of  $M_n$ consisting of all matrices $B$ 
		with pairwise different eigenvalues 
		$(\lambda_1, \ldots, \lambda_m,\, \mu_1, \ldots, \mu_k)$ of the multiplicities
		$\alpha_1,\ldots, \alpha_{m},\,\beta_1,\ldots, \beta_{k}$, correspondingly, 
		satisfying the conditions
		$$
		\dim(\text{Ker }(B - \lambda_i I)) > 1,\, 1\leq i \leq m, \text{ and }
		\dim(\text{Ker }(B - \mu_j I)) = 1,\, 1 \leq j \leq k.
		$$
		Then 
		the following statements are true.
		
		\begin{enumerate} 
			\item  If $m\leq 1$, then all matrices in $\cal M$ are integrable. 
			
			\item   If\, $2 \leq m \leq n - M + 1$,  then $\cal M$ contains both integrable and non-integrable matrices.

			\item  If $m > n - M + 1$, then all matrices in $\cal M$ are non-integrable.

		\end{enumerate} 
		In particular, $\cal M$ contains an integrable matrix if and only if $m\le n-M+1$.
		
	\end{theorem}
	
	Let us mention that applying Theorem~\ref{THM:matrix_int_classification} for  $m=0$ we obtain the result of \cite{IntegratorsOfMatricies} that each non-derogatory matrix is integrable. On the other hand, for diagonalizable matrices, Theorem~\ref{THM:matrix_int_classification} implies the main result of~\cite{DanielyanGuterman}, see Theorem~\ref{THM:recall_diag_matrix_int_classification}.
	
	The paper is organized as follows. 
	In Section 3,   we firstly recall some  
	properties of plane trees,  Shabat polynomials, and conservative polynomials used in the paper. In particular, we discuss the relations between these classes of polynomials and  plane trees. Then, we prove  Theorem~\ref{THM:main}. 
	In Section 4, we prove Theorems \ref{Cor:THM:general_integrable_iff_fullintegral} and \ref{THM:matrix_int_classification} and some additional results. In particular, we prove   that if a matrix is  diagonalizable and integrable, then each matrix   with the same spectrum  is integrable.

	\section{$S$-full integrals of polynomials}
	\subsection{Plane trees}

	We recall that a {\it tree} is a connected graph without cycles,
	and a {\it plane tree} is a tree  embedded into the plane.
	Two plane trees $\lambda_1, \lambda_2$ are called equivalent
	if there exists an orientation preserving homeomorphism $\mu$
	of the plane such that $\mu(\lambda_1) = \lambda_2$.
	A trivial induction shows that a tree with $n$ edges has $n+1$ vertices.
	Let $\lambda$ be a plane tree and $(\gamma_1,\gamma_2,\dots, \gamma_{n+1})$
	the sequence  of valencies of vertices of $\lambda$.
	Since $\lambda$ has no loops, every edge of $\lambda$
	is adjacent exactly to two vertices of $\lambda$, implying that
	
	\begin{equation} \label{ga}
		\sum_{i=1}^{n+1}\gamma_i=2n.
	\end{equation}
	We will refer to this fact by saying that
	$(\gamma_1,\gamma_2,\dots, \gamma_{n+1})$ is a {\it partition} of $2n.$

	The following two lemmas are known (see Section 1.5.2 and Section 1.6.1 of
	\cite{lz} for more details and generalizations).
	We provide full proofs of these results since their ideas
	are of importance for our proof
	of Theorem \ref{THM:main}.
	
	\begin{lemma} \label{lemder1} Let $n$ be a positive integer.
		Then for any partition $(\gamma_1,\gamma_2,\dots, \gamma_{n+1})$ of $2n$
		there exists a plane tree $\lambda$ with $n$ edges and the sequence
		of valencies of vertices $(\gamma_1,\gamma_2,\dots, \gamma_{n+1})$.
	\end{lemma}
	\begin{proof} The proof is by induction on $n$. For $n=1$ the statement is
		clearly true. To prove the inductive step, we observe that for $n>1$ 
		equality \eqref{ga} implies that at least one element of
		$(\gamma_1,\gamma_2,\dots, \gamma_{n+1})$ is equal to 1 and
		at least one does not. Assuming that $\gamma_1=1,$ $\gamma_2>1$,
		let us consider the partition $(\gamma_2-1,\gamma_3,\dots ,\gamma_{n+1})$
		of the number $2(n-1).$ By the induction assumption, there exists a tree
		$\lambda'$ with   $n-1$ edges and the sequence of valencies of vertices
		$(\gamma_2-1,\gamma_3,\dots ,\gamma_{n+1})$. To obtain now a required tree, 
		$\lambda$ it is enough to glue an extra edge to the vertex of valency
		$\gamma_2-1$ of $\lambda'.$
	\end{proof}
	
	In this paper, instead of ordinary plane trees, we  consider
	bicolored plane trees.
	By definition, a {\it bicolored plane tree} is a plane tree whose
	vertices are colored in
	black and white colors in such a way that any edge connects vertices of
	different colors. Two bicolored plane trees $\lambda_1, \lambda_2$
	are called {\it  equivalent}
	if they are equivalent as plane trees and the corresponding homeomorphism $\mu$
	preserves the  colors of vertices. Any plane tree $\lambda$
	can be ``bicolored'' by choosing one of two possible colorings for
	an arbitrary vertex of $\lambda$
	and expanding coloring to the remaining vertices.
	
	One can easily see that if $(\alpha_1, \ldots, \alpha_p)$ and
	$(\beta_1, \ldots, \beta_q)$ are the sequences of valencies of white
	and black vertices of a bicolored plane tree with $n$ edges,
	then the equalities
	\begin{equation} \label{pq1}
		\sum_{i = 1}^{p} \alpha_j = \sum_{i = 1}^{q} \beta_i = n
	\end{equation}
	and
	\begin{equation} \label{pq2}
		p + q = n + 1
	\end{equation}
	holds. In turn, the analogue of Lemma \ref{lemder1} is the following statement.
	
	\begin{lemma}  \label{lemder2}
		Let $n$ be a positive integer.
		Then for any partitions
		$(\alpha_1, \ldots, \alpha_p)$ and
		$(\beta_1, \ldots, \beta_q)$ of  $n$ such that  $p + q = n + 1$
		there exists a bicolored plane tree $\lambda$ with $n$ edges
		and the valency sequences of white and black vertices 
		$(\alpha_1, \ldots, \alpha_p)$ and
		$(\beta_1, \ldots, \beta_q)$.
	\end{lemma}
	\begin{proof}
		As above, the proof is by induction on $n$. For $n=1$ the statement is  true.
		To prove the inductive step we observe that
		if $n>1$, then \eqref{pq1} and \eqref{pq2} still imply that 
		at least one of the numbers $\alpha_j,$ $\beta_i$ is equal to one. Furthermore, if, say, $\alpha_1=1$, then 
		it follows from \eqref{pq1} and \eqref{pq2}
		that at least one of the numbers $\beta_i$, say $\beta_1$, is greater than one.
		By the induction assumption, there exists a bicolored plane tree $\lambda'$
		with $n-1$ edges and the sets of valencies of white and black vertices
		$(\alpha_2, \ldots, \alpha_p)$ and
		$(\beta_1-1, \ldots, \beta_q)$.
		Gluing now an extra edge to the vertex of valency $\beta_1-1$ of $\lambda'$,
		we obtain a required tree $\lambda$. If instead of  $\alpha_1$
		one of the numbers $\beta_i$ is equal to one,
		then the proof is obtained by an obvious modification.
	\end{proof}

	The main result of this section, which is used in the proof of
	Theorem \ref{THM:main}, is the following variation of the above lemmas.
	
	\begin{lemma} \label{lemder3}
		Let $n$ be a positive integer and $(\gamma_1,\gamma_2,\dots, \gamma_{n+1})$ be
		a partition  of $2n$. Assume that for an integer $l,$ $1\leq l\leq n,$
		the inequality
		\begin{equation} \label{gam}
			\gamma_{1}+\gamma_{2}+\dots +\gamma_{l}\leq n
		\end{equation}
		holds. Then there exist $p\ge l$
		and a bicolored plane tree  $\lambda$  with $n$ edges
		and the sequences of valencies of white and black vertices
		$(\alpha_1, \ldots, \alpha_p)$ and
		$(\beta_1, \ldots, \beta_q)$ such that:
		$$
		(\gamma_1,\gamma_2,\dots, \gamma_{n+1})=
		(\alpha_1, \ldots, \alpha_p,\beta_1, \ldots, \beta_q).
		$$
	\end{lemma}
	\begin{proof} The proof is  by induction on $n$. 
		For $n=1$ the statement is  true. Furthermore, it follows from Lemma~\ref{lemder1} that if $l=1$, then it is true for any $n$,  
		since any plane tree can be bicolored. 
		
		To prove the inductive step in case $l>1$, let us observe first that conditions of the theorem imply that at least one of the following conditions   holds: 
		
		(a):   the sequence
		$ (\gamma_1,\gamma_2,\dots, \gamma_{l})$ contains 1 and
		$ (\gamma_{l+1},\gamma_{l+2},\dots, \gamma_{n+1})$
		contains an element different from 1, 
		
		(b):   the sequence $ (\gamma_{l+1},\gamma_{l+2},\dots, \gamma_{n+1})$
		contains 1 and
		$ (\gamma_1,\gamma_2,\dots, \gamma_{l})$
		contains an element different from 1.
		
		Indeed, $1$ belongs to the set $(\gamma_1,\gamma_2,\dots, \gamma_{n+1})$ and hence belongs to at least one of the sets  $(\gamma_1,\gamma_2,\dots, \gamma_{l})$ and $(\gamma_{l+1},\gamma_{l+2},\dots, \gamma_{n+1})$. 
		If $1\in (\gamma_1,\gamma_2,\dots, \gamma_{l})$
		and $(\gamma_{l+1},\gamma_2,\dots, \gamma_{n+1})$
		contains a non-unit, then (a) takes place.
		On the other hand, if all elements of the set  
		$(\gamma_{l+1},\gamma_{l+2},\dots, \gamma_{n+1})$
		are units, then necessarily at least one of the elements of the set 
		$(\gamma_1,\gamma_2,\dots, \gamma_{l})$
		is not a unit, so that (b) takes place.
		In case $1\in (\gamma_{l+1},\gamma_{l+2},\dots, \gamma_{n+1})$,
		the proof is similar.
		
		In case (a) holds,  
		the inductive step goes as follows. 
		Without loss of generality we may assume that
		$\gamma_1=1$ and $\gamma_{l+1}>1.$ 
		Defining now a partition of the integer $2(n-1) $ by the formula
		
		\begin{equation} \label{parti}
			(\gamma_2,\dots,\gamma_{l},\gamma_{l+1}-1,\gamma_{l+2},\dots , \gamma_{n+1})
		\end{equation}
		and observing that \eqref{gam} implies that
		
		\[
		\gamma_2+\gamma_{3}+\dots +\gamma_{l}\leq n-1,
		\]
		we conclude  
		by the induction assumption that there exists a bicolored plane tree
		$\lambda'$ with $n-1$ edges, whose sequence of valencies of vertices
		coincides with \eqref{parti}, and whose sequence of valencies
		of white vertices ``contains''
		$(\gamma_2,\dots, \gamma_{l})$.
		Therefore, gluing an extra edge to the vertex of valency $\gamma_{l+1}-1$ of
		$\lambda'$  we  obtain  a required tree~$\lambda$.  In case (b) holds,  the proof  
		is obtained by an obvious modification.
	\end{proof}

	\subsection{Shabat polynomials.}
	Let $P(z)$ be a complex polynomial. We recall that zeros
	$w_1, \ldots, w_n$ of $P'(x)$
	are called (finite) {\it critical points} of $P(z)$ and the values 
	$P(w_1),\ldots,P(w_n)$ are called
	(finite) {\it critical values} of~$P(z).$

	\begin{defin}
		A complex polynomial $P(z)$ is called a {\it Shabat polynomial}
		if it has at most
		two (finite) critical values. Two Shabat polynomials $P_1(z), P_2(z)$
		are called {\it equivalent} if
		there exist polynomials $\mu_1$ and $\mu_2$ of degree one such that
		$P_2= \mu_1\circ P_1 \circ \mu_2 :=\mu_1(P_1(\mu_2))$. 
	\end{defin}
	
	Notice that by choosing an appropriate polynomial $\mu_1$ it is always
	possible to assume that critical values of $P(z)$ are $0$ and $1$.
	
	The following statement is a particular case of the correspondence
	between Belyi pairs and ``dessins d'enfants"
	(see \cite{lz}, \cite{sz}, \cite{ShabVoe} for more detail).

	\begin{theorem}  \label{THM:shabat_tree}
		There is a bijective correspondence between the equivalence classes of Shabat
		polynomials and the equivalence classes of bicolored plane trees.
	\end{theorem}
	We briefly recall how
	this correspondence is constructed. 
	Let $P(x)$ be a
	Shabat polynomial with critical values $0$ and $1$.
	Then the corresponding plane tree $\lambda_P$ is  defined as the preimage
	$\lambda_P=P^{-1}([0,1])$ of the segment $[0,1]$ with respect to the function
	$P(x)\,:\,\C\rightarrow \C$. By definition, white (resp., black)
	vertices of $\lambda_P$ are preimages of the point $0$ (resp., of the point 1)
	and edges of $\lambda_P$ are preimages of the segment~$[0,1]$.

	In the other direction, if   $\lambda$ is a bicolored tree with $n$ edges
	and the sequences
	of valencies of white and  of black vertices 
	$\alpha_1,\ldots,\alpha_p$ and
	$\beta_1,\ldots,\beta_q,$ correspondingly, then
	the corresponding Shabat polynomial $P(x)\in \C[x]$
	with critical values 0 and 1 is defined by 
	the conditions
	\[
	\begin{cases}
		P(x) = c(x - a_1)^{\alpha_1}\ldots(x - a_p)^{\alpha_p}\\
		P(x) - 1 = c(x - b_1)^{\beta_1}\ldots(x - b_q)^{\beta_q},
	\end{cases}
	\]
	where $a_1,\ldots, a_p,\,b_1,\ldots,b_q \in \C$ are pairwise distinct and $c\in \C$ is distinct from zero. 
	Thus, a system that determines
	a Shabat polynomial of a tree is a  system of polynomial equations with the unknowns
	$a_1,\ldots, a_p,\,b_1,\ldots,b_q ,c$ obtained 
	from equating coefficients of like terms   
	in the equality
	\begin{equation} \label{sys}
		c(x - a_1)^{\alpha_1}\ldots(x - a_p)^{\alpha_p}
		- 1 = c(x - b_1)^{\beta_1}\ldots(x - b_q)^{\beta_q}. 
	\end{equation}
	Notice that there could be
	{\it several} trees with the  same sequences
	$\alpha_1,\ldots,\alpha_p$ and $\beta_1,\ldots,\beta_q$.
	All corresponding Shabat polynomials satisfy the same system~\eqref{sys}. 
	
	After fixing critical values of a Shabat polynomial, we still have a ``degree of freedom'' corresponding to a choice of  
	$\mu_2$. Thus, we can impose some further restrictions on
	system \eqref{sys}. For example, we can assume that $a_1=0$ and $b_1=1$.
	Theorem \ref{THM:shabat_tree}  implies that in this case the system
	\eqref{sys} has only {\it finitely many} solutions.
	Since  \eqref{sys}
	provide us with equations in $a_1,\ldots, a_p,\,b_1,\ldots,b_q, c$
	with {\it rational } and even integer coefficients, this implies that
	solutions are necessarily {\it algebraic numbers}.
	Thus, for any plane tree the corresponding equivalence class of 
	Shabat polynomials contains polynomials with algebraic coefficients (see  \cite{lz}, \cite{sz}).

	Theorem \ref{THM:shabat_tree} combined with Lemma \ref{lemder3}
	allows us to prove
	the following statement, which is used for the proof
	of Theorem~\ref{THM:main}.

	\begin{cor} \label{COR:shabat_with_certain_mult}
		Let  $m,k\ge 1$   and $\alpha_1,\ldots, \alpha_{m}\geq 2,\,
		\beta_1,\ldots, \beta_{k}\ge 1$ be  integers, and 
		\[
		n=\sum\limits_{j=1}^{m}\alpha_j+\sum\limits_{i=1}^{k}\beta_i, \ \ \ \
		M=\sum\limits_{j=1}^{m}\alpha_j.
		\]
		Assume that $m \leq n - M + 1$.
		Then there exist pairwise distinct
		$a_1,\ldots, a_m,\, b_1,\ldots, b_k \in \overline{\mathbb Q}
		\subseteq \mathbb K$
		and a Shabat polynomial $P(x)\in \overline{\mathbb Q}[x]
		\subseteq \mathbb K[x]$ of
		degree $n+1$  such that
		\begin{equation} \label{con1}
			P'(x)=(x - a_1)^{\alpha_1}\ldots(x - a_m)^{\alpha_m}
			(x - b_1)^{\beta_1}\ldots(x - b_k)^{\beta_k}
		\end{equation}
		and 
		\begin{equation} \label{con2}
			P(a_1) = \ldots = P(a_m)=0. 
		\end{equation}
	\end{cor}
	\begin{proof} Let us set
		\begin{equation} \label{gam2}
			(\gamma_1,\gamma_2,\dots, \gamma_{n+2})=
			(\alpha_1 + 1,\ldots, \alpha_m + 1, \beta_1 + 1, \ldots, \beta_k + 1,
			\underbrace{1, \ldots, 1}_{n - (m+k)+2}).
		\end{equation}
		It is easy to see that \eqref{gam2} is a partition of $2(n+1)$.
		Indeed, the definition of $n$ implies that $n\geq m+k$,
		so that \eqref{gam2} is well-defined. In addition,
		\[
		\sum\limits_{j=1}^{m}(\alpha_j+1) +
		\sum\limits_{i=1}^{k}(\beta_i+1) +
		n - (m+k)+2 =
		\sum\limits_{j=1}^{m}\alpha_j+\sum\limits_{i=1}^{k}\beta_i+n +2=2n+2.
		\]
		Since there are  $n+2$ elements in \eqref{gam2}  and
		\[
		\sum\limits_{j=1}^{m}(\alpha_j+1)=M+m\leq n+1
		\]
		by the condition,
		it follows from Lemma \ref{lemder3} that there exists a bicolored plane tree
		$\lambda$ with $n+1$ edges and sequences of white and black valencies
		$(\mu_1, \ldots, \mu_p)$ and
		$(\nu_1, \ldots, \nu_q)$, where $p\geq m,$  such that
		\[
		(\gamma_1,\gamma_2,\dots, \gamma_{n+2}) =
		(\mu_1, \ldots, \mu_p, \nu_1, \ldots, \nu_q), 
		\]
		\[
		(\alpha_{1}+1,\dots, \alpha_{m}+1)=(\mu_{1}, \ldots, \mu_m), 
		\]
		and
		\[
		( \beta_1 + 1, \ldots, \beta_k + 1,
		\underbrace{1, \ldots, 1}_{n - (m+k)+2})=(\mu_{m+1}, \ldots, \mu_p, \nu_1, \ldots, \nu_q).
		\]
		Applying   Theorem \ref{THM:shabat_tree}, we see that there exists
		a Shabat polynomial $\widetilde P(z)\in \overline{\mathbb Q}[z]$ of degree $n+1$ such that
		\begin{equation} \label{f1}
			\widetilde P(x) = c(x - x_1)^{\mu_1}\ldots(x - x_p)^{\mu_p}
		\end{equation}
		and
		\[
		\widetilde P(x) - 1 = c(x - y_1)^{\nu_1}\ldots(x - y_q)^{\nu_q}.
		\]
		for some pairwise distinct $x_1,\ldots, x_p,\, y_1,\ldots, b_q \in \overline{\mathbb Q}$
		and $0\neq c \in \overline{\mathbb Q}.$ 
		
		By construction, among the points $x_{m+1}, \dots,x_p,y_1,y_2,\dots,y_q$ there are exactly $k$ points that are zeros of $P'(x)$. 
		Denoting these points by $z_1,z_2,\dots ,z_k$ 
		and setting  
		$$
		(a_1,a_2,\dots,a_m)=(x_1,x_2,\dots,x_m) \ \ \mbox{ and }  \ \ 
		(b_1,b_2,\dots,b_k)= (z_1, z_2, \dots, z_k), 
		$$
		we see that
		$\widetilde P'(x)$ is divisible by
		\[
		(x - a_1)^{\alpha_1}\ldots(x - a_m)^{\alpha_m}
		(x - b_1)^{\beta_1}\ldots(x - b_k)^{\beta_k}.
		\]
		Moreover, since the degree of the last polynomial is $n$, the equality
		\begin{equation} \label{f2}
			\widetilde P'(x)=c(n+1)(x - a_1)^{\alpha_1}\ldots(x - a_m)^{\alpha_m}
			(x - b_1)^{\beta_1}\ldots(x - b_k)^{\beta_k}
		\end{equation}
		holds. Finally, it follows from \eqref{f1} and \eqref{f2} that for
		the Shabat polynomial
		\[
		P(x)=\frac{\widetilde P(x)}{c(n+1)}
		\]
		equalities \eqref{con1} and \eqref{con2} hold, and   $ P(x)\in \overline{\mathbb Q}[x].$
	\end{proof}
	
	\subsection{Conservative polynomials}
	In addition to a Shabat polynomial, with every plane tree one can
	associate a polynomial of a different type, described
	by the following definition.
	
	\begin{defin}
		A complex polynomial $C(x)$ is called {\it conservative} if all
		its critical points are fixed, that is,
		if the equality $C^{\prime}(\zeta)=0,$ $\zeta\in \C,$
		implies that $C(\zeta)=\zeta.$
		A conservative polynomials $C(x)$ is called {\it normalized}
		if $C(x)$ is monic and $C(0)=0.$
		Two conservative polynomials $C_1(x)$ and $C_2(x)$
		are called {\it equivalent} if
		there exists a complex polynomial $\mu$ of degree one
		such that $C_2=\mu^{-1}\circ C_1\circ \mu.$
	\end{defin}
	
	Conservative polynomials were introduced by
	Smale \cite{smale} in connection with his ``mean value conjecture".
	Motivated by Smale's conjecture Kostrikin proposed in \cite{kostr}
	several conjectures concerning conservative polynomials.
	In particular, he conjectured that the
	number of normalized conservative polynomials of degree $n$
	is finite and is equal to $C_{2n-2}^{n-1}.$
	This
	conjecture was proved by Tischler in the paper \cite{tishl}.
	In fact, he proved the following statement, implying the  Kostrikin conjecture, see \cite[Theorem 4.2]{tishl}.

	\begin{theorem}  \label{THM:tish}
		There is a bijective correspondence between the equivalence classes of
		conservative polynomials of degree $n$ and the equivalence classes
		of bicolored plane trees
		with $n-1$ edges.
	\end{theorem}
	
	For  a conservative polynomial $C$, the corresponding plane tree
	$\lambda_C$ is constructed as follows (see \cite{tishl}
	for more detail and \cite{pakovich} for some pictures).
	Let $\zeta$ be a critical point of $C(x)$ and $d\geq 2$ the local multiplicity
	of $C(x)$ at $\zeta.$ Then one can show that the immediate attractive
	basin $B_{\zeta}$ of $\zeta$ is a disk and that
	there is an analytic conjugation of $C(x)$ on $B_{\zeta}$
	to $x\rightarrow x^{d}$ on the unit disk $D$ such that the conjugating map $\phi_{\zeta} : \mathbb D \rightarrow
	B_{\zeta}$ extends continuously to the closed unit
	disk $\overline {\mathbb D}$. Let $U$ be a union of $d-1$ radial
	segments which are forward invariant under the map $x\rightarrow x^{d}$
	on $\overline {\mathbb D}$, and $U_{\zeta}$ the image of $U$
	under the map $\phi_{\zeta}$, considered as a bicolored graph
	with a unique white vertex, which is the image of zero, and $d-1$
	black vertices, which are the images of end-points of $U$.
	In this notation, $\lambda_C$ is defined as a union
	$\lambda_C=\cup_{i=1}^p U_{\zeta_i},$ where $\zeta_i,$ $1\leq i \leq p,$ are all
	finite critical points of $C(x).$
	Note that by construction $\lambda_C$ is a forward
	invariant of $C(x)$, and white (resp. black) vertices of $\lambda_C$
	are attractive (resp. repelling) fixed points of $C(x).$
	
	In the other direction, if $\lambda$ is a bicolored plane tree with $n-1$
	edges and the sequence of valencies
	of white vertices $\alpha_1,\ldots, \alpha_p$, then a corresponding
	conservative polynomial $C(x)$ satisfies the system
	\begin{equation} \label{sys2}
		\begin{cases}
			C'(x) = c(x - c_1)^{\alpha_1}\ldots(x - c_p)^{\alpha_p}\\
			C(c_i) = c_i, 
		\end{cases}
	\end{equation}
	where $c_1,\ldots,c_p \in \C$  are pairwise distinct and  $c\in \C$ is distinct from zero. 
	
	Notice that in distinction with system \eqref{sys} 
	the valencies of black vertices do not appear in system \eqref{sys2}. In addition, 
	the number of edges of a tree corresponding to a conservative polynomial
	of the degree $n$ is $n-1$ instead of $n$. Nevertheless, 
	similar to system \eqref{sys}, system \eqref{sys2} reduces to
	a system of equations in $c_1,\ldots, c_p, c$
	with rational coefficients. Furthermore,  
	if $C(x)$ is normalized, then 
	the number of solutions of \eqref{sys2} is finite and these solutions are
	algebraic numbers. Thus, for any plane tree the corresponding equivalence class of 
	conservative polynomials contains polynomials with algebraic coefficients.

	A counterpart of
	Corollary \ref{COR:shabat_with_certain_mult}, which follows from 
	Theorem \ref{THM:tish} is the following statement.

	\begin{cor} \label{COR:conservative_with_certain_mult}
		Let $l\geq 1$ and $\gamma_1,\ldots, \gamma_{l}\ge 1$ be  integers,
		and $n = \sum\limits_{i=1}^{l}\gamma_i.$
		Then there exist pairwise distinct
		$c_1,\ldots,c_l\in \overline{\mathbb Q} \subseteq \mathbb K$
		and a conservative polynomial
		$C(x)\in\overline{\mathbb Q}[x] \subseteq \mathbb K[x]$
		of degree $n+1$  such that
		\begin{equation} \label{sa}
			C'(x)=(x - c_1)^{\gamma_1}\ldots(x - c_l)^{\gamma_l} 
		\end{equation}
		and 
		\begin{equation} \label{as}
			C(c_i)=c_i, \ \ \ \ 1\leq i \leq l.
		\end{equation}
	\end{cor}
	\begin{proof}
		Let $(\delta_1,\ldots,\delta_{n+1-l})$ be an arbitrary partition of
		the number $n$ containing $n+1-l$ elements. For example,
		we can take
		\[
		(\delta_1,\ldots,\delta_{n+1-l})=(l,\underbrace{1, \ldots, 1}_{n - l}).
		\]
		By Lemma \ref{lemder2}, there exists  a bicolored plane tree $\lambda$
		with $n$ edges and the sequences of white and black valencies
		$(\gamma_1,\ldots,\gamma_{l})$ and $(\delta_1,\ldots, \delta_{n+1-l}).$
		Therefore, by Theorem \ref{THM:tish}, there exist
		pairwise distinct $\widetilde c_1,\ldots,\widetilde c_l \in \overline{\mathbb Q}$ and a conservative polynomial
		$\widetilde C(x)\in\overline{\mathbb Q}[x]$ of degree $n+1$ such that
		\[
		\widetilde C'(x) = c(x - \widetilde c_1)^{\gamma_1}\ldots(x - \widetilde c_l)^{\gamma_l}\\
		\]
		for some $0\ne c\in \overline{\mathbb Q}.$  Setting now $\mu=\varepsilon x$, where 
		$\varepsilon$ satisfies $\varepsilon^{n-1}=1,$
		we see that the conservative polynomial $C=\mu^{-1}\circ \widetilde C\circ \mu$ has algebraic coefficients and 
		satisfies~\eqref{sa} and \eqref{as} for $c_i=\widetilde c_i/\varepsilon $, $1\leq i \leq l.$  
	\end{proof}
	

	\subsection{Proof of Theorem \ref{THM:main}}

	Since the condition 
	$S=\emptyset$ provides no restrictions, the first part of the theorem is trivially true.  The second part is also true since for any
	polynomial $F(x)\in {\mathbb K}[x]$ such that
	
	\[
	F'(x) = f(x) =
	(x - a)^{\alpha_1}(x - b_1)^{\beta_1}\ldots(x - b_k)^{\beta_k},
	\]
	the polynomial
	$F(x) - F(a)$ obviously is an $S$-full integral of $f(x)$ for $S=\{a\}.$ 
	
	To prove the fourth part, we observe that  if  $F(x)$ is
	an $S$-full integral of $f(x),$ then 
	\[
	(x - a_i)^{\alpha_i + 1}\,|\,F(x), \ \ \ i = 1,\ldots, m.
	\]
	Therefore, 
	\[
	\deg F(x) \geq M + m,  
	\]
	implying that 
	$$
	m \leq  \deg F(x)- M =n+1-M.
	$$
	
	Let us prove now the  third part. Notice that the condition 
	\begin{equation} \label{uslovie}
		2\leq m\leq n-M+1
	\end{equation}
	implies that  $k>0$, for otherwise $n=M$ and \eqref{uslovie} leads to a contradictory inequality $2\leq m \leq 1.$ 
	Since $k>0$ and $m\geq 2$, it follows from  
	Corollary \ref{COR:shabat_with_certain_mult} that 
	there exist pairwise distinct $a_1,\ldots, a_m,$ $b_1,\ldots, b_k \in
	\overline{\mathbb Q} \subseteq\mathbb K$
	and a Shabat polynomial
	$P(z)\in\overline{\mathbb Q}[z] \subseteq \mathbb K[z]$ 
	of degree $n+1$  such that the equalities \eqref{con1} and \eqref{con2} hold. Thus, $P(z)$ is an $S$-full
	integral of $P'(z)$ for $S=\{a_1,\ldots, a_m\}$, and hence 
	the first statement of the third part is true.
	
	Further, since $k>0$ and $m\geq 2$ imply that $k+m\geq 1$, we can apply  Corollary \ref{COR:conservative_with_certain_mult} for
	\[
	(\gamma_1,\ldots,\gamma_{l})=
	(\alpha_1,\ldots, \alpha_{m},\beta_1,\ldots, \beta_{k}),
	\]
	and find  pairwise distinct
	$c_1, \dots ,c_l\in\overline{\mathbb Q} \subseteq \mathbb K$
	and a conservative polynomial
	$C(z)\in \overline{\mathbb Q}[z] \subseteq \mathbb K[z]$
	of degree $n+1$ such that
	\[
	C'(z)= (x - c_1)^{\alpha_1}\ldots(x - c_m)^{\alpha_m}
	(x - c_{m+1})^{\beta_1}\ldots(x - c_{m+k})^{\beta_{k}}
	\]
	and
	\begin{equation} \label{ci}
		C(c_i) = c_i, \ \ \ 1\leq i \leq m+k.
	\end{equation}
	Since any primitive $F(z)$ of $C'(z)$ has the form 
	$F(z) = C(z) + c,$ $c\in \mathbb K,$ it follows from \eqref{ci} 
	that
	$C'(z)$
	does not have an
	$S$-full integral for any subset $S$ of $\{c_1,\ldots,c_{m+k}\}$ that contains at least two elements. 
	Thus, to prove the second statement of the third part of the theorem, we can set for example 
	$$
	a_i=c_i, \ \ \ 1\leq i \leq m, \ \ \ \ b_i=c_{i+m},\ \ \ 1\leq i \leq k.
	$$
	\qed
	
	\begin{corollary}
		For all values of multiplicities $\alpha_1,\ldots, \alpha_{m}\geq 2,$ $\beta_1,\ldots, \beta_{k}\ge 1$  satisfying  $$2 \leq m \leq n - M + 1$$
		there exist infinitely many polynomials that have $S$-full integrals and infinitely many polynomials that do not have $S$-full integrals. 
	\end{corollary} 
	
	\begin{proof}
		If $P$ is a  polynomial,  $S=\{a_1,\ldots, a_m\}$  is a set, and $\mu=az+b$ is a non-constant linear map, then $P$  
		has an $S$-full integral if and only if the polynomial $P\circ \mu$ has  an $\widetilde S$-full integral for the set $\widetilde S=\{\mu^{-1}(a_1),\ldots, \mu^{-1}(a_m)\}$. Then Theorem \ref{THM:main} implies the result.    
	\end{proof}
	
	\section{$S$-full integrals and  matrix integrability}

	In this section, we prove Theorem \ref{Cor:THM:general_integrable_iff_fullintegral} and Theorem \ref{THM:matrix_int_classification}. Notice that our proof of Theorem \ref{Cor:THM:general_integrable_iff_fullintegral} is effective: it  
	shows how a matrix integral is constructed
	from an $S$-full integral of the
	characteristic polynomial. In this relation, we remark that some integrable matrices may have infinitely many integrals. For more detail, we refer the reader to~\cite{IntegratorsOfMatricies}.
	
	Let  $E_{ij}\in M_k$ be $(i,j)$-th matrix unit, i.e., the matrix with $1$ on the $(i,j)$-th position and 0 elsewhere,  $J_k=E_{12}+E_{23}+\ldots +E_{k-1k}\in M_k$   the Jordan matrix of the size~$k$ if $k\ge 2$, and $J_1=0\in M_1$. We denote by  $X^{\top}$   the transposed matrix and by $p_X(x)=\det (xI-X)$   the characteristic polynomial of~$X\in M_n$. Finally, we denote by   $\diag(X_1,\ldots, X_n)$   the block-diagonal matrix with the blocks $X_1, \ldots , X_n$. 
	
	\subsection{Proofs of Theorems \ref{Cor:THM:general_integrable_iff_fullintegral} and \ref{THM:matrix_int_classification}}
	The following lemma is proved in \cite[Lemma 7]{IntegratorsOfMatricies}.
	
	\begin{lemma} 
		If $B \in M_n$ has an integral $A$, then  $\left( \begin{smallmatrix} X & 0^{\intercal}\\ 0 & 1 \end{smallmatrix} \right) A \left( \begin{smallmatrix} X^{-1} & 0^{\intercal}\\ 0 & 1 \end{smallmatrix} \right) $ is an integral of $XBX^{-1}$ for any invertible $X\in M_n$.
	\end{lemma}
	
	The following corollary is straightforward:
	\begin{cor} \label{CorSimInt}
		1. If  $B \in M_n$ is integrable, then  for any invertible $X\in M_n$ the matrix $XBX^{-1}$ is integrable.  
		
		2. If $B $ has an integral $A$ and the matrix $X$ commutes with $B$, then $\left( \begin{smallmatrix} X & 0^{\intercal}\\ 0 & 1 \end{smallmatrix} \right) A \left( \begin{smallmatrix} X^{-1} & 0^{\intercal}\\ 0 & 1 \end{smallmatrix} \right) $ is also an integral of~$B$.
	\end{cor}
	
	\begin{lemma} \label{LM:jordan_basic_for_integral} 
		For any non-zero vector $v \in\K^n$, let $k$ be the smallest positive integer such that $v_k\ne 0$, i.e., $v= (\underbrace{0,\ldots, 0}_{k-1}, v_k, v_{k+1}, \ldots, v_n)$, where $1\le k\le n,$ $v_k\ne 0$.
		Then there exists a polynomial $h(x) \in\K[x]$ of the degree at most $(n - k)$
		such that $h(J_n)$ is invertible and
		$v\cdot h(J_n) =
		(\underbrace{0,\ldots, 0}_{k-1}, 1, \underbrace{0,\ldots, 0}_{n - k}).$
	\end{lemma}
	\begin{proof}
		It is well-known that
		
		\[
		h(J_n) = \begin{pmatrix}
			h(0) & \frac{h'(0)}{1!} &\frac{h''(0)}{2!} &\frac{h'''(0)}{3!}
			& \ldots & \frac{h^{(n-2)}(0)}{(n-2)!}& \frac{h^{(n-1)}(0)}{(n-1)!}\\
			0&h(0) & \frac{h'(0)}{1!}&\frac{h''(0)}{2!} & \ldots &
			\frac{h^{(n-3)}(0)}{(n-3)!}&\frac{h^{(n-2)}(0)}{(n-2)!}\\
			0&0&h(0) & \frac{h'(0)}{1!} & \ldots & \frac{h^{(n-4)}(0)}{(n-4)!}
			& \frac{h^{(n-3)}(0)}{(n-3)!}\\
			\cdots&\cdots&\ddots&\ddots& 
			&\cdots&\cdots\phantom{\vdots}
			\\
			\cdots&\cdots& &\ddots&\ddots&\cdots&\cdots\phantom{\vdots}\\ 
			0&0& \ldots   & \ldots& 0 &h(0) & \frac{h'(0)}{1!}\\
			0 & 0&  \ldots & \ldots &0& 0 & h(0)
		\end{pmatrix}
		\]
		and a direct computation shows that  
		
		\[
		v \cdot h(J_n) = \left(\underbrace{0,\ldots, 0}_{k-1},\
		v_k h(0),\ v_k\frac{h'(0)}{1!} + v_{k+1}h(0),\
		\ldots,\ \sum\limits_{i = k}^{n}v_i
		\frac{h^{(n - i)}(0)}{(n - i)!} \right).
		\]
		Thus, the equality
		$v\cdot h(J_n) =
		(\underbrace{0,\ldots, 0}_{k-1}, 1,
		\underbrace{0,\ldots, 0}_{n - k})$
		is equivalent   to the  system
		
		\begin{equation} \label{eq:ys}
			\begin{cases}
				h(0) = \frac{1}{v_k}\\
				h'(0) = -\frac{1!}{v_k}v_{k+1}h(0)\\
				h''(0) = -\frac{2!}{v_k}
				\left(v_{k+2}h(0) + v_{k+1}\frac{h'(0)}{1!}\right)\\
				\vdots\\
				h^{n - k}(0) = -\frac{(n - k)!}{v_k}
				\sum\limits_{i = k+1}^{n}v_i\frac{h^{(n - i)}(0)}{(n - i)!}
			\end{cases}
		\end{equation}
		Writing   $h(x)$ as
		$h(x) = c_0 + c_1 x + \ldots + c_{n-k}x^{n-k}$ with   unknown coefficients $c_0, \ldots, c_{n-k}$,  we see that
		\eqref{eq:ys} transforms into the triangular linear system for
		$c_0,\ldots,c_{n-k}$, 
		which is solvable since the determinant of this system is equal to
		$\prod_{l = 0}^{n - k} l!\ne 0$.
		Moreover $\det(h(J_n)) = (h(0))^n = \frac{1}{v_k^n} \neq 0$
		and hence $h(J_n)$ is invertible as desired.
	\end{proof}
	
	\begin{defin} 
		Let $B \in M_n$. Any matrix $A\in M_{n+1}$ of the form
		$A = \left(
		\begin{smallmatrix}
			B & u^{\intercal}\\
			v & b
		\end{smallmatrix}
		\right)$
		is called an {\it integral extension} of the matrix $B$.
	\end{defin}
	Notice that an integral extension is not necessarily an integral of the matrix $B$, but any integral of $B$ is an integral extension by the definition. Note also that if a matrix $A$ is an integral of $B$ then necessarily $b = \frac{tr(B)}{n}$. Indeed, $b=tr(A)-tr(B)$. On the other hand,   $p_A'(x) = (n+1)p_B(x)$ implies   $tr(A)\cdot n=  (n+1)\, tr(B)$.
	
	By Corollary~\ref{CorSimInt}(1), 
	if a matrix $B$ is integrable, then all matrices similar to $B$ are integrable as well. Thus,  we  may  assume that
	the matrix $B \in M_n$ is in the Jordan normal form. Namely, 
	
	\begin{equation} \label{eq:B}
		B=\diag(B_1,\ldots,B_m),
	\end{equation}
	where $B_i $ is the   Jordan block for $b_i,$
	i.e. the union of all Jordan cells of $B$ with the eigenvalue $b_i$ ordered in non-increasing order of block sizes.
	We denote the number of Jordan cells in $B_i$ by $\beta_i$,  and   the
	sizes of the Jordan cells of $B_i$ by
	$k_{i,1} \geq \ldots \geq k_{i,\beta_i}$.
	Then
	
	\begin{equation}
		B_i=\diag(b_iI_{k_{i,1}} + J_{k_{i,1}} 
		, \ b_iI_{k_{i,2}} + J_{k_{i,2}},\ \ldots,\ b_iI_{k_{i,\beta_i}} + J_{k_{i,\beta_i}})  \in M_{\alpha_i}
	\end{equation}
	where
	\begin{equation} \label{EQ:alpha}
		\alpha_i= k_{i,1}+\ldots +k_{i,\beta_i}, \ i = 1,\ldots, m.
	\end{equation}
	Note that in the introduced notations the characteristic polynomial of $B$ is 
	$$
	p_B(x) = (x - b_1)^{\alpha_1}\ldots(x - b_m)^{\alpha_m},
	$$
	where $b_1,\ldots,b_m$ are pairwise distinct.

	Let $A = \left(
	\begin{smallmatrix}
		B & u^{\intercal}\\
		v & b
	\end{smallmatrix}
	\right)\in M_{n+1}$ be an integral extension of $B$. Then in accordance with the introduced notations 
	\begin{equation} \label{eq:A}
		A =
		\left( \begin{smallmatrix}
			b_1I_{k_{1,1}} + J_{k_{1,1}}&&&&&&&(u^{1, 1})^\top\\
			&\ddots&&&&&&\vdots\\
			&&b_1I_{k_{1,\beta_1}} + J_{k_{1,\beta_1}} &&&&&(u^{1, \beta_1})^\top\\
			&&&\ddots &&&& \vdots\\
			&&&&b_mI_{k_{m,1}} + J_{k_{m,1}} &&& (u^{m, 1})^\top\\
			&&&&&\ddots&&\vdots\\
			&&&&&&b_mI_{k_{m,\beta_m}} + J_{k_{m,\beta_m}} & (u^{m, \beta_m})^\top\\
			v^{1, 1}& \ldots& v^{1, \beta_1} & \ldots & v^{m,1} & \ldots
			& v^{m, \beta_m} & b
		\end{smallmatrix} \right) \in M_{n+1},
	\end{equation}
	where $v^{i,j}, u^{i,j}\in \K^{k_{i,j}}.$
	
	\begin{defin} \label{DEF:matrix_jordan_form}
		In the above notations if
		\begin{equation} \label{eq:rij}
			v^{i,j} = (\underbrace{0, \ldots, 0}_{r_{i,j}}, 1, 0, \ldots, 0) 
			\in \K^{k_{i,j}}
		\end{equation}
		for some $r_{i,j}$, $0 \leq r_{i,j} \leq k_{i,j},$
		then the matrix $A$ is called a
		{\it normalized} integral extension of~$B$.
	\end{defin}
	Notice that  in case $r_{i, j} = k_{i, j}$
	we have $v^{i, j} = 0.$

	\begin{lemma} \label{lem:normA}
		Let
		$A = \left(
		\begin{smallmatrix}
			B & u^{\intercal}\\
			v & b
		\end{smallmatrix}
		\right)\in M_{n+1}$ be an integral extension of   $B \in M_n$. Then there exists a normalized integral  extension $\tilde A\in M_{n+1}$ of   $B$  such that the matrices $A$ and $\tilde A$ are similar and the similarity matrix  $C=\diag(C_1,1)$ satisfies $C_1B=BC_1$.
	\end{lemma}
	\begin{proof} Let $B$ and $A$ be determined by the equalities \eqref{eq:B} and \eqref{eq:A}, correspondingly. 
		For each pair $(i,j)$, $1\le i\le m,\ 1\le j\le \beta_{i}$, we consider the vector $ {v}^{i,j}\in \K^{k_{i,j}}$ and define a matrix $C_{i,j}\in M_{k_{i,j}}$ as follows. If ${v}^{i,j}=0$, then we set $C_{ij} =I_{k_{i,j}}$. If ${v}^{i,j}\ne 0$, then by Lemma
		\ref{LM:jordan_basic_for_integral}  there exists $h_{i,j}\in \K_{k_{i,j}}[x]$ such that $h_{i,j}(J_{k_{i,j}})$ is invertible and  $v_{i,j} \cdot h_{i,j}(J_{k_{i,j}})$  has only one nonzero entry. In this case, we set $C_{ij} =h_{i,j}(J_{k_{i,j}})$. Note that $C_{i,j}$ commutes with $b_iI_{k_{i,j}} + J_{k_{i,j}}$ since the matrix $J_{k_{i,j}}$ commutes with a polynomial of itself. Hence,
		$$
		B\cdot \diag(C_{1,1},\ldots, C_{m,\beta_m}) = \diag(C_{1,1},\ldots, C_{m,\beta_m})\cdot B.
		$$
		Let us consider the matrix
		\[
		C =  \diag(C_{1,1},\ldots, C_{m,\beta_m},1) \in M_{n+1}
		\]
		and the matrix $\tilde A=CAC^{-1}$.
		By the choice of the blocks $C_{i,j}$, we have:
		\[
		\tilde{A} =  
		\left( \begin{smallmatrix}
			b_1I_{k_{1,1}} + J_{k_{1,1}}&&&&&&&(\tilde{u}^{1, 1})^\top\\
			&\ddots&&&&&&\vdots\\
			&&b_1I_{k_{1,\beta_1}} + J_{k_{1,\beta_1}} &&&&&(\tilde{u}^{1, \beta_1})^\top\\
			&&&\ddots &&&& \vdots\\
			&&&&b_mI_{k_{m,1}} + J_{k_{m,1}} &&& (\tilde{u}^{m, 1})^\top\\
			&&&&&\ddots&&\vdots\\
			&&&&&&b_mI_{k_{m,\beta_m}} + J_{k_{m,\beta_m}} & (\tilde{u}^{m, \beta_m})^\top\\
			\tilde{v}^{1, 1}& \ldots& \tilde{v}^{1, \beta_1} & \ldots &
			\tilde{v}^{m,1} & \ldots & \tilde{v}^{m, \beta_m} & b 
		\end{smallmatrix} \right)\in M_{n+1},
		\]
		where $\tilde{v}^{i,j} = (\underbrace{0, \ldots, 0}_{r_{i,j}}, 1, 0, \ldots, 0)
		\in \K^{k_{i,j}},$ as required.
	\end{proof}

	
	\begin{corollary} \label{Cor:ExN}
		Assume that
		$B \in M_n$ is integrable, and let $A\in M_{n+1}$ be its integral. Then there exists an integral $\tilde{A}\in M_{n+1}$ of   $B$ such that $\tilde A$ is a normalized integral  extension  of~$B$ and $p_{\tilde{A}}(x) = p_{A}(x).$
	\end{corollary}
	\begin{proof}
		Let $A\in M_{n+1}$ be an integral of $B$. By Lemma \ref{lem:normA},  
		there exists a normalized integral  extension $\tilde{A}$ of~$B$. By Corollary~\ref{CorSimInt}(2), the matrix $\tilde{A}$ is an integral of $B$ since $ \diag(C_{1,1},\ldots, C_{m,\beta_m})$ commutes with $B$. Finally,  $p_{A}(x) = p_{\tilde{A}}(x)$ since the matrices $A$ and $\tilde{A}$ are similar. 
	\end{proof}
	
	
	We denote by $\K(y)$  the field of formal rational functions in the variable $y$ over the field~$\K$.
	
	\begin{lemma}
		Let $k \geq 1$, $0 \leq r \leq k$ be integers, and
		$X = (yI_k - J_k) \in M_{k}(\K(y))$. Then for the vector
		$v = (\underbrace{0, \ldots, 0}_{r}, 1, \underbrace{0, \ldots, 0}_{k - r - 1})\in \K^{k} $ and an arbitrary vector $u = (u_1, \ldots, u_k) \in \K^{k}$  it holds that 
		\begin{equation} \label{eq:vXu}
			vX^{-1}u^{\top} = \sum_{t = 1}^{k - r} u_{t+r}y^{-t}.
		\end{equation}
	\end{lemma}
	\begin{proof} We use the notation $J_k^0=I_k.$ Note that in case $k = r$ the vector $v$ is the zero vector. Otherwise
		the direct multiplication
		\[
		(yI_k - J_k)\sum_{t = 1}^{k } y^{-t}J^{t-1}_{k} = 
		\sum_{t = 1}^{k } (y^{1-t}J^{t-1}_{k} - y^{-t}J^{t}_{k})  = I_k - y^{-k}J^{k}_k = I_k
		\]
		shows that $X^{-1} = \sum_{t = 1}^{k} y^{- t}J^{t - 1}_{k}.$  
		
		Therefore,
		$$
		vX^{-1} =
		\left(\underbrace{0, \ldots, 0}_{r}, y^{-1}, y^{-2}, \ldots, y^{-(k - r)}\right)
		$$
		since it is the $(r+1)-$th row of $X^{-1}$.
		Hence, we obtain:
		$$
		(vX^{-1})u^{\top} = \sum_{t = 1}^{k - r} u_{t+r}y^{-t},
		$$
		as desired.
	\end{proof}
	
	\begin{lemma} \label{LM:general_pA}
		Assume $B\in M_n$ is in the Jordan normal form,
		and let $A = \left(
		\begin{smallmatrix}
			B & u^{\intercal}\\
			v & b
		\end{smallmatrix}
		\right)\in M_{n+1} $ be its normalized integral extension. Then in
		the notation  \eqref{eq:B} -- \eqref{eq:rij}  
		the characteristic polynomial of~$A$ is
		\begin{equation} \label{eq:pA}
			p_A(x) = (x - b)p_B(x) -
			\sum\limits_{i = 1}^{m}\sum\limits_{j = 1}^{\beta_i}
			\sum\limits_{t=1}^{k_{i,j} - r_{i,j}}
			u_{t+r_{i,j}}^{i,j}\frac{p_B(x)}{(x - b_i)^{t}}.
		\end{equation}
		
	\end{lemma}
	\begin{proof}
		Using the formula for the determinant of a block matrix with invertible block $X_1$
		\[
		\det
		\left(
		\begin{smallmatrix}
			X_1&X_2\\X_3&X_4
		\end{smallmatrix}
		\right)
		= \det(X_1)\det(X_4 - X_3 X_1^{-1} X_2),
		\]
		we obtain that
		\[
		p_{A}(x) = \det(xI_{n + 1} - A) = \det
		\left(
		\begin{smallmatrix}
			xI_n - B & -u^{\intercal}\\
			-v & x - b
		\end{smallmatrix}
		\right)
		= \det(xI_n - B)\det((x - b) - v(xI_n - B)^{-1}u^{\top}),
		\]
		and therefore
		\begin{equation} \label{eq:pApB_block}
			p_{A}(x) = p_{B}(x)\left((x - b) - v(xI_n - B)^{-1}u^{\top}\right).
		\end{equation}
		Considering   $(xI_n - B)$   as an element of $M_n(\K(x))$,  
		in the notation \eqref{eq:B} -- \eqref{eq:rij}   we obtain:
		\[
		(xI_n - B)^{-1}
		=\left[xI_n - \diag(b_1 I_{k_{1,1}} + J_{k_{1,1}},\ \ldots,\ b_m I_{k_{m,\beta_m}} + J_{k_{m,\beta_m}})\right]^{-1},
		\]
		\[
		(xI_n - B)^{-1}
		= \left[\diag((x - b_1)I_{k_{1,1}} - J_{k_{1,1}},\ \ldots,\ (x - b_m)I_{k_{m,\beta_m}} - J_{k_{m,\beta_m}})\right]^{-1},
		\]
		\[
		(xI_n - B)^{-1}
		= \diag(\left[(x - b_1)I_{k_{1,1}} - J_{k_{1,1}}\right]^{-1},\ \ldots,\ \left[(x - b_m)I_{k_{m,\beta_m}} - J_{k_{m,\beta_m}}\right]^{-1}).
		\]
		Hence,
		\begin{equation} \label{eq:v(xI-B)u}
			v(xI_n - B)^{-1}u^{\top} = \sum\limits_{i = 1}^{m}\sum\limits_{j = 1}^{\beta_i} v^{i,i}\left[(x -b_i)I_{k_{i,j}} - J_{k_{i,j}}\right]^{-1} (u^{i,j})^{\top}.
		\end{equation}
		Now we apply the formula \eqref{eq:vXu} with $k = k_{i,j},\, r = r_{i,j},\, y = (x - b_i)$ to each summand and obtain that
		\begin{equation} \label{eq:v(xI-Bi)u}
			v^{i,i}\left[(x -b_i)I_{k_{i,j}} - J_{k_{i,j}}\right]^{-1} (u^{i,j})^{\top} = \sum\limits_{t = 1}^{k_{i,j} - r_{i,j}} u_{t + r_{i,j}}^{i,j}(x - b_i)^{-t}.
		\end{equation}
		It  remains  to substitute the expressions \eqref{eq:v(xI-B)u} and \eqref{eq:v(xI-Bi)u} into the equality \eqref{eq:pApB_block}, and the result follows.
	\end{proof}

	\begin{lemma} \label{LM:fg'=F}
		Let $f(x),\, g(x)\in \K[x]$ be polynomials, and $\deg g(x)> \deg f(x)$. Then for any $t \in\K$  satisfying
		$f(t) \neq 0$ and for any $k$, $0\le k\le  \deg g(x) - \deg f(x)  ,$   there exists a polynomial $h(x)\in \K[x]$, $\deg(h(x))\le k,$ such that
		\begin{equation} \label{eq:fh}
			(fh)^{(i)}(t) = g^{(i)}(t),\ i = 0, \ldots, k.
		\end{equation}
	\end{lemma}
	\begin{proof}
		Let us write $h(x) = a_0 + a_1(x - t) + \ldots + a_k(x - t)^k$ with   unknown coefficients $a_0, \ldots, a_k\in \K$. Then
		$h^{(i)}(t) = i!a_i, i = 0, \ldots, k.$ Further,
		\[
		(fh)^{(s)}(t) = \sum\limits_{i = 0}^{s} f^{(s - i)}h^{(i)}(t) =
		\sum\limits_{i = 0}^{s} i!a_if^{(s - i)}(t).
		\]
		Thus, the condition \eqref{eq:fh}
		is equivalent to the equality
		\[
		\begin{pmatrix}
			0!f(t) & 0 & 0 & \ldots & 0\\
			0!f'(t) & 1!f(t) & 0 &\ldots & 0\\
			\vdots&\vdots&\vdots&\ddots&\vdots\\
			0!f^{(k)}(t) & 1!f^{(k-1)}(t) & 2!f^{(k-2)}(t) & \ldots & k!f(t)
		\end{pmatrix}
		\begin{pmatrix}
			a_0\\
			a_1\\
			\vdots
			\\
			a_k
		\end{pmatrix} = \begin{pmatrix}
			g(t)\\
			g'(t)\\
			\vdots
			\\
			g^{(k)}(t)
		\end{pmatrix}
		\]
		The determinant of this linear system is equal to
		$0!f(t)\cdot 1!f(t)\cdot \ldots\cdot k!f(t).$ Hence   the matrix of the system is invertible, since $f(t) \neq 0$.
		Then \eqref{eq:fh} is satisfied for
		\[
		\begin{pmatrix}
			a_0\\
			a_1\\
			\vdots
			\\
			a_k
		\end{pmatrix} = \begin{pmatrix}
			0!f(t) & 0 & 0 & \ldots & 0\\
			0!f'(t) & 1!f(t) & 0 &\ldots & 0\\
			\vdots&\vdots&\vdots&\ddots&\vdots\\
			0!f^{(k)}(t) & 1!f^{(k-1)}(t) & 2!f^{(k-2)}(t) & \ldots & k!f(t)
		\end{pmatrix}^{-1}
		\begin{pmatrix}
			g(t)\\
			g'(t)\\
			\vdots
			\\
			g^{(k)}(t)
		\end{pmatrix}. 
		\]
	\end{proof}
	
	\begin{proof}[Proof of Theorem~\ref{Cor:THM:general_integrable_iff_fullintegral}.]
		In the notation \eqref{eq:B} -- \eqref{eq:rij},
		we set $S = \{b_i\ |\ \beta_i > 1\}$.
		
		1. First, we show that if $A$ is an integral of $B$, then $p_{A}(x)$ is
		an $S$-full integral of $(n+1)p_{B}(x).$ 
		By  Corollary \ref{Cor:ExN}, we can assume that $A$ is normalized.
		
		Since $A$ is an integral of $B$, the equality   $p'_{A}(x)=(n+1)p_{B}(x)$ follows from the definition. It remains to show that $p_A(x)$ is an $S$-full integral, i.e., that  $p_A(b_i) = 0$ for all $b_i \in S.$  
		Observe that  $\beta_i > 1$ since  $b_i \in S$. Hence, there are at least two summands in the  equality \eqref{EQ:alpha} for $\alpha_i$. Since $k_{i,j}>0$,
		we obtain
		\[
		\alpha_i = \sum\limits_{j = 1}^{\beta_i} k_{i,j} >
		\max\limits_{1 \leq j \leq \beta_i}(k_{i,j}).
		\]
		Since $t$ ranges from  $1$ till $(k_{i,j}-r_{i,j})$ in the decomposition  \eqref{eq:pA},
		it follows that  
		\[
		t \leq \max\limits_{1\leq j \leq \beta_i}(k_{i,j} - r_{i,j}) \leq
		\max\limits_{1\leq j \leq \beta_i}(k_{i,j})  < \alpha_i.
		\]
		Since $b_i$ is the zero of the polynomial $p_B(x)$ of multiplicity $\alpha_i$,  we have
		$$
		(x-b_i) \mid \frac{p_B(x)}{(x - b_i)^{t}}\in \K[x]
		$$
		for each $t=1,\ldots ,(k_{i,j} - r_{i,j}).$ 
		Thus,   every summand in the decomposition \eqref{eq:pA} of $p_A(x)$
		is divisible by $(x - b_i)$, and hence $p_A(b_i) = 0.$
		Therefore, $p_A(x)$ is an $S$-full integral of $(n+1)p_B(x)$ as desired. 
		
		2. Now let us assume that $F(x)$ is an $S$-full integral of $(n+1)p_B(x)$ and prove that there exist vectors $u,v\in \K^n$ such that for  $A=  \left(
		\begin{smallmatrix}
			B & u^{\intercal}\\
			v &  b
		\end{smallmatrix}
		\right)\in M_{n+1} $, where $b=\frac{tr(B)}{n}$,
		the equality  $p_A(x) = F(x)$ holds. Let us observe first   that to prove the theorem it is  enough to find $u,v $ such that  
		\begin{equation} \label{eq:*}
			p_A^{(j)}(b_i) = F^{(j)}(b_i),\ i = 1, \ldots, m;\,
			j = 0, \ldots \alpha_i - 1.
		\end{equation}
		Indeed, \eqref{eq:*} implies that the polynomial $(p_A - F)(x)$ has
		$\sum \limits_{i = 1}^{m} \alpha_i = n$ zeros
		counting with the multiplicities. Notice that the coefficient
		at $x^{n-1}$ of $p_{B}(x)$ is equal to $-tr(B).$ Hence, since $F'(x) = (n+1)p_{B}(x)$,   the coefficient
		at $x^{n}$ of $F(x)$ is equal to $-\frac{n+1}{n}tr(B)$. On the other hand, the coefficient
		at $x^{n}$ of $p_{A}(x)$ is equal to
		$$
		-tr(A) = -(tr(B) +  b ) = -\frac{n+1}{n}tr(B).
		$$
		Thus,
		since both $p_A(x)$ and $F(x)$ are monic and the coefficient
		at $x^{n}$ of both polynomials is equal to $\frac{n+1}{n}tr(B),$ we obtain that
		$\deg((p_A - F)(x)) \leq n + 1 - 2 = n - 1.$
		Therefore,   $(p_A - F)(x)\equiv 0$, and so $p_A(x) = F(x)$ as desired.
		
		Now, let us prove \eqref{eq:*}.

		2.1. At first, we consider $b_i\in S$. Let us set $v^{i,j} = u^{i,j} =  0 \in \K^{k_{i,j}}$  for each  $j=1,\ldots, \beta_i$ and each $i=1, \ldots , m$ satisfying $b_i\in S$. Then by the formula \eqref{eq:pA} we obtain that 
		$(x - b_i)^{\alpha_i}\ | \ p_A(x)$ for any values of the other coordinates of the vectors $u,v$. 
		Observe that by definition of $F(x)$ and properties of zeros of derivatives  for any $b_i \in S$ it holds that
		$(x - b_i)^{\alpha_i}\ |\ F(x)$  or, equivalently, $  F^{(j)}(b_i) = 0,
		\, 0 \leq j < \alpha_i.$ 
		Hence, for any $b_i \in S$ we obtain
		\begin{equation} \label{EQ:zeros_from_S}
			p_A^{(j)}(b_i) = F^{(j)}(b_i) = 0,\, 0 \leq j < \alpha_i.
		\end{equation}
		
		2.2. Now let us consider  an eigenvalue $b_i$ of $B$ such that $b_i \notin S.$  According to the notation \eqref{eq:B} --- \eqref{eq:rij}  this means that   $\beta_i = 1$ and
		$k_{i, 1} = \alpha_i$.
		Let us set $f(x) = \frac{p_{B}(x)}{(x - b_i)^{\alpha_i}}$. Then $f(x)\in \K^{n-\alpha_i}[x]$ is a polynomial with the property $f(b_i)\ne 0$. Since
		$$
		\deg(F(x)) - \deg(f(x)) = n + 1 - (n - \alpha_i) =
		\alpha_i + 1 > \alpha_i - 1 \ge 0 ,
		$$
		it follows from  Lemma \ref{LM:fg'=F} that the system of equations
		\[
		\left(f(x)h(x)\right)^{(l)}(b_i) = F^{(l)}(b_i),
		\  0\leq l < \alpha_i, 
		\]
		on the coefficients of a polynomial $h(x)  $ has a solution $h_0(x) \in \K_{\alpha_i - 1}[x]$ of degree  $\deg h_0(x)= q$ with $0\le q\le  \alpha_i - 1 $. 
		
		Now we define the  elements $w_t\in \K$, $t=0,\ldots, \alpha_i-1$, as follows. Let us expand the polynomial $h_0(x)$ on the degrees of $x-b_i$, i.e., $h_0(x) = \sum_{t = 0}^{q} w_t(x - b_i)^t$. This expansion defines $w_0, \ldots, w_q.$ For $q+1\le t\le  \alpha_i - 1$,  we  set $w_t=0$.
		Let us show that for $v^{i,1} = (1, 0,\ldots, 0)$
		and $u^{i,1}_{t+1} = w_{{\alpha_i} - t}$ it holds that
		\begin{equation} \label{EQ:zeros_not_from_S}
			p_A^{(l)}(b_i) =
			\left(f(x)h_0(x)\right)^{(l)}(b_i) = F^{(l)}(b_i) ,\ 0\leq l < \alpha_l.
		\end{equation}
		Indeed, let us consider 
		\begin{equation} \label{eq:eq}
			\tilde p_A(x)=\sum\limits_{j = 1}^{\beta_i}
			\sum\limits_{t=1}^{k_{i,j} - r_{i,j}}
			u_{t+r_{i,j}}^{i,j}\frac{p_B(x)}{(x - b_i)^{t}}.
		\end{equation}
		Then by the decomposition \eqref{eq:pA} 
		\[
		p_A(x)+\tilde p_A(x)=(x - b)p_B(x) - \sum\limits_{\begin{smallmatrix} \varphi = 1, \ldots,m \\ \varphi \ne i \end{smallmatrix}} \sum\limits_{j = 1}^{\beta_\varphi}
		\sum\limits_{t=1}^{k_{\varphi,j} - r_{\varphi,j}}
		u_{t+r_{\varphi,j}}^{\varphi,j}\frac{p_B(x)}{(x - b_\varphi)^{t}}, 
		\]
		and using the fact that $(x - b_{i})^{\alpha_i}\,|\, p_B(x)$, we obtain 
		\begin{equation} \label{eq:div}
			(x - b_{i})^{\alpha_i} \mid (p_A(x)+\tilde p_A(x)  ).
		\end{equation}
		Since $\beta_i = 1$ and $k_{i,1} = \alpha_i$, we can remove the first summation in \eqref{eq:eq} and simplify the second as follows: 
		\[
		\tilde p_A(x) = \sum\limits_{t=1}^{k_{i,1} - r_{i,1}}
		u_{t+r_{i,1}}^{i,1}\frac{p_B(x)}{(x - b_i)^{t}} = \sum\limits_{t=1}^{\alpha_i - r_{i,1}}
		u_{t+r_{i,1}}^{i,1}\frac{p_B(x)}{(x - b_i)^{t}} .
		\]
		By the choice of $v^{i,1},u^{i,1}$, we have $r_{i, 1} = 0$ and $u^{i,1}_{t} = w_{\alpha_i - t}$. Therefore,
		\begin{multline*}
			\tilde p_A(x) = \sum\limits_{t=1}^{\alpha_{i}}
			u_{t}^{i,1}\frac{p_B(x)}{(x - b_i)^{t}} = \sum\limits_{t=1}^{\alpha_{i}}
			w_{\alpha_i - t}\frac{p_B(x)}{(x - b_i)^{t}}   = \sum\limits_{s=0}^{\alpha_{i} - 1}
			w_{s}\frac{p_B(x)}{(x - b_i)^{\alpha_{i} - s}} \\ = \sum\limits_{s=0}^{\alpha_{i} - 1} 
			w_{s} (x - b_i)^{  s}\frac{p_B(x)}{(x - b_i)^{\alpha_{i} }} = (f\cdot h_0)(x), \phantom{aaaaaaaaaaaaaaaaaaaaaaaaaa}
		\end{multline*}
		where the last equality  in the first row is due to the substitution $s=\alpha_i-t$. 
		Now \eqref{eq:div} implies that
		$$
		p_A^{(l)}(b_i) =
		\left(f\cdot h_0\right)^{(l)}(b_i) ,\ 0\leq l < \alpha_l.
		$$
		Therefore, by the choice of $h_0(x)$ we get
		\[
		p_A^{(l)}(b_i) =
		\left(f\cdot h_0\right)^{(l)}(b_i) = F^{(l)}(b_i) ,\ 0\leq l < \alpha_l.
		\] 
		Combining the equalities
		\eqref{EQ:zeros_from_S} and \eqref{EQ:zeros_not_from_S}
		we obtain   \eqref{eq:*}, and the result follows.
	\end{proof}
	
	
	\begin{proof}[Proof of Theorem~\ref{THM:matrix_int_classification}.]
		By Theorem \ref{Cor:THM:general_integrable_iff_fullintegral}, the matrix  $B$ is
		integrable if and only if
		an integral $\int p_B(x) {\rm d}x$ takes the same value on all elements of~$S= \{ \lambda_1, \ldots, \lambda_m\},$ which 
		means that $p_{B}(x)$ has an $S$-full integral.
		It  remains  to  apply
		Theorem \ref{THM:main} to get the desired statement.
	\end{proof}

	Notice that  Theorem~\ref{THM:matrix_int_classification} implies the following result proved in \cite[Theorem~3.13]{DanielyanGuterman}.
	\begin{theorem} \label{THM:recall_diag_matrix_int_classification}
		Let  $m,k\ge 0$ and
		$\alpha_1,\ldots, \alpha_{m}\geq 2$
		be  integers,
		$
		n=k + \sum\limits_{j=1}^{m}\alpha_j,
		$
		and 
		$\cal M$ the subset of  $M_n$ consisting of
		diagonalizable matrices $B$ 
		with pairwise different eigenvalues 
		$(\lambda_1, \ldots, \lambda_m,\, \mu_1, \ldots, \mu_k)$ of the multiplicities
		$\alpha_1,\ldots, \alpha_{m},\,1,\ldots, 1$, correspondingly.
		
		Then $\cal M$ contains an integrable matrix if and only if $m\le k+1$,  and contains a non-integrable matrix if and only if  $m\ge 2$. Moreover,
		\begin{enumerate} 
			\item If $m\leq 1$, then all matrices in $\cal M$ are integrable. 
			
			\item   If\, $2 \leq m \leq k + 1$,  then $\cal M$ contains both integrable and non-integrable matrices.

			\item  If $m > k + 1$, then all matrices in $\cal M$ are non-integrable. 
			
		\end{enumerate} 
	\end{theorem}
	
	\begin{proof}
		If $B$ is diagonalizable, then,
		in the notations of Theorem \ref{THM:matrix_int_classification},
		we have  $\beta_1 = \ldots = \beta_k = 1$
		and hence $n - M + 1 = k + 1.$ Thus, Theorem \ref{THM:matrix_int_classification} provides the desired result.
	\end{proof}
	
	\subsection{Supplements}
	
	The main aim of this section is to prove the following theorem:
	\begin{theorem} \label{THM:diag_integral_is_general_integral}
		Let $B \in M_n$ be diagonalizable, and  $B' \in M_n$ have the same eigenvalues as $B$ counting the multiplicities, i.e.,
		$p_{B'}(x) = p_{B}(x).$ If $B$ is integrable, then $B'$ is integrable.  
	\end{theorem}
	
	Since by  \cite[Corollary 5.2]{DanielyanGuterman}  an integrable diagonalizable matrix   has a diagonalizable integral, Theorem~\ref{THM:diag_integral_is_general_integral} follows from the lemma below.
	

	\begin{lemma}
		Let $B \in M_n$ be diagonalizable, and let $B' \in M_n$ have the same eigenvalues as $B$ counting the multiplicities, i.e.,
		$p_{B'}(x) = p_{B}(x).$ If $B$ is integrable and  $A=\left(
		\begin{smallmatrix}
			B & u^{\intercal}\\
			v & b
		\end{smallmatrix}
		\right)$
		is a diagonalizable integral of $B,$ then there exists $X \in GL_{n}(\K)$ and there exists an integral $A'$ of $XB'X^{-1}$ such that $A'= \left(
		\begin{smallmatrix}
			XB'X^{-1} & u^{\intercal}\\
			v & b
		\end{smallmatrix}
		\right)$ with the same vectors~$(u, v)$ and an element $b$ as~$A$ has.
	\end{lemma}
	\begin{proof}
		Without loss of generality, we assume that
		\[
		B = \diag(\underbrace{b_1, \ldots, b_1}_{\alpha_1},\ldots,
		\underbrace{b_m, \ldots, b_m}_{\alpha_m}).
		\]
		By Lemma~\ref{lem:normA}, without loss of generality we may assume that $A$ is a normalized integral extension of~$B$. 
		Since $B$ is diagonalizable, in the notations of the formula \eqref{eq:A} we have that $k_{i,j}=1$ for all $i,j$. Thus $\beta_i=\alpha_i,$ and $v^{i , j} , u^{i, j} \in \K$.
		By the criterion of diagonalizability of an integral \cite[Theorem 4.1]{DanielyanGuterman},
		we obtain that if
		$\alpha_i > 1$, then  
		$v^{i , j} = u^{i, j} = 0,\ j = 1, \ldots, \alpha_i.$
		By   Lemma \ref{LM:general_pA},  these equalities imply that
		\[
		p_{A}(x) = (x - b)p_{B}(x) -
		\sum\limits_{\substack{s=1  : \\ \alpha_s = 1}}^{m} u_{s, 1}
		\frac{p_{B}(x)}{x - b_s}.
		\]

		Let $B''$ be the Jordan normal form of $B'$ such that the order of the   diagonal elements of $B''$ coincides with the order of the diagonal elements of $B$. Let us consider $A' = \left(
		\begin{smallmatrix}
			B'' & u^{\intercal}\\
			v & b
		\end{smallmatrix}
		\right).$ 
		Since $A$ is a normalized extension of $B$  and $v^{i , j} = u^{i, j} = 0,\ {\color{blue}j = 2}, \ldots, \alpha_i,$ it follows that $A'$ is  a normalized extension of $B''$ by definition. Hence, by   Lemma \ref{LM:general_pA}, we obtain
		\[
		p_{A'}(x) = (x - b)p_{B''}(x) -
		\sum\limits_{\substack{s=1 : \\ \alpha_s = 1}}^{m} u_{s, 1}
		\frac{p_{B''}(x)}{x - b_s}.
		\]
		Since $p_{B''}(x) =p_{B'}(x) = p_{B}(x)$, this implies 
		$p_{A'}(x) = p_{A}(x)$, and therefore 
		$$
		p_{A'}'(x) = (n +1)p_{B}(x) = (n +1)p_{B''}(x).
		$$
		Thus, $A'$ is an integral of~$B''.$
	\end{proof}
	
	\begin{rmk}
		Even if $A=\left(
		\begin{smallmatrix}
			B & u^{\intercal}\\
			v & b
		\end{smallmatrix}
		\right)$ is a diagonalizable integral of $B$, the integral
		$A'=\left(
		\begin{smallmatrix}
			B' & u^{\intercal}\\
			v & b
		\end{smallmatrix}
		\right)$
		of $B'$ can be non-diagonalizable.
		Consider, for example, $B = \left(
		\begin{smallmatrix}
			a & 0\\
			0 & a
		\end{smallmatrix}
		\right),\ B' = \left(
		\begin{smallmatrix}
			a & 1\\
			0 & a
		\end{smallmatrix}
		\right).$
		Then $p_{B}(x) = p_{B'}(x)$ and the vectors $u=v=(0, 0)$ provide a diagonal integral
		$\left(
		\begin{smallmatrix}
			a & 0 & 0\\
			0 & a & 0\\
			0 & 0 & a
		\end{smallmatrix}
		\right)$
		of $B.$
		However, $\left(
		\begin{smallmatrix}
			a & 1 & 0\\
			0 & a & 0\\
			0 & 0 & a
		\end{smallmatrix}
		\right)$ is not diagonalizable.
	\end{rmk}

	\begin{cor} \label{cor_cor}
		Consider an arbitrary (not necessarily diagonalizable) $B \in M_n.$
		If $p_{B}(x)$ has an $S$-full integral, where $S$ is the set of all multiple zeros of $p_B(x)$, then $B$ is integrable.
	\end{cor}
	\begin{proof}
		If $p_{B}(x)$ has a full integral, then a diagonal matrix $B'$ with the same spectrum as $B$ is  integrable by Theorem \ref{Cor:THM:general_integrable_iff_fullintegral}. Hence, the matrix $B$ is integrable  by  Theorem~\ref{THM:diag_integral_is_general_integral}.
	\end{proof}
	
	\begin{rmk}
		The assertion opposite to Corollary~\ref{cor_cor} is not true. Consider the non-diagonalizable
		matrix
		$$
		B =
		\left(
		\begin{smallmatrix}
			1&0&0&0\\
			0&1&0&0\\
			0&0&-1&1\\
			0&0&0&-1
		\end{smallmatrix}\right).
		$$
		Since $p_{B}(x) = (x - 1)^2(x + 1)^2$, it follows from  \cite[Theorem
		2.14]{DanielyanGuterman} that the matrix $B'=\diag (1,1,-1,-1)$ is not integrable.
		Nevertheless, the matrix
		$$
		A = \left(\begin{smallmatrix}
			1&0&0&0&0\\
			0&1&0&0&0\\
			0&0&-1&1&0\\
			0&0&0&-1&\frac{4}{3}\\
			0&0&1&1&0
		\end{smallmatrix}\right)
		$$
		is an integral of $B$, since
		$p_{A}(x) = (x - 1)^2(x^3 + 2x^2 - \frac{1}{3}x + \frac{8}{3}) =
		x^5 - \frac{10}{3}x^3 + 5x - \frac{8}{3}$ and
		\[
		p_{A}'(x) = 5(x^4 - 2x^2 + 1) = 5(x - 1)^2(x + 1)^2 = 5p_{B}(x).
		\]
	\end{rmk}
	
	As a conclusion, we made the following observation:
	\begin{rmk}
		Consider an ordered set of matrices  $\{B_1,\ldots, B_m\}\subset M_n$ having 
		the same eigenvalues counting the multiplicities. Let us assume that the following conditions are satisfied:
		\[
		\dim(\text{Ker}(B_i - \lambda I)) \geq \dim(\text{Ker}(B_{i+1} - \lambda I)) 
		\]
		for any $\lambda$, and 
		\[
		\dim(\text{Ker}(B_i - \lambda I)) > \dim(\text{Ker}(B_{i+1} - \lambda I))
		\]
		for at least one value of~$
		\lambda,
		$ i.e., the number of Jordan cells in a Jordan block does not increase with the growth of $i$, and decreases for at least one block for each $i$.
		In this case, similarly to the proof of Theorem \ref{Cor:THM:general_integrable_iff_fullintegral}, it can be shown that  
		if $B_{i}$ is integrable for a certain index $i$, then  $B_{i+1}$ is integrable.
		Since the number of cells in a certain block decreases, we obtain just 1 cell on a certain step. Therefore,  that if $m$ is big enough, then $B_m$ is non-derogatory, and hence integrable. 
		Thus, if $B_1$ is not integrable, then there exists a positive integer $k$ such that
		the matrices $B_1,\ldots,B_k$ are not integrable, but the matrices
		$B_{k + 1},\ldots, B_m$ are integrable.
	\end{rmk}
	

	\begin{rmk} We notice that the main results of this paper can be extended to an arbitrary algebraically closed field $\widehat{\K}$ of zero characteristic, not necessarily contained in $\C$.  Indeed, since the proof of Theorem~\ref{Cor:THM:general_integrable_iff_fullintegral} uses only  linear algebra techniques, its conclusion  
		remains true over  $\widehat{\K}$.
		The same is true for the parts 1, 2, and 4 of Theorem~\ref{THM:main}. To prove the part 3,  we constructed special polynomials $f(x)\in \overline{\mathbb Q}[x] $, where $\overline{\mathbb Q}$ was considered as a subset of $\K$ and $\C.$ 
		However, an arbitrary algebraically closed field of characteristic zero $\widehat{\K}$ contains  $\mathbb Q$, and  therefore it also contains a subfield $\overline{\mathbb Q}'$ isomorphic to $\overline{\mathbb Q}$.   It is straightforward to check that the image of a polynomial $f(x)$ from  $\overline{\mathbb Q}[x] $ in $\overline{\mathbb Q}'[x]$ has the same properties of $S$-full integrability as $f(x)$. Therefore, Theorem~\ref{THM:main} 
		remains true over $\widehat{\K}$, and the same is true for 
		Theorem~\ref{THM:matrix_int_classification} obtained as a combination of Theorem~\ref{Cor:THM:general_integrable_iff_fullintegral} and Theorem~\ref{THM:main}. 
	\end{rmk} 
	
	\section*{Acknowledgment}
	
	The authors are grateful to the referee for the useful suggestions.
	
	The research of the first author was supported by ISF Grant No. 1994/20. The research of the third and the forth authors was supported by ISF Grant No. 1432/18.

\end{document}